\documentclass[12pt]{amsart}
\usepackage{amscd}
\usepackage{amssymb}

\def\NZQ{\Bbb}               

\def\ZZ{{\NZQ Z}}

\def\PP{{\NZQ P}}

\def\frk{\frak}               

\def\pp{{\frk p}}

\def\qq{{\frk q}}

\def\mm{{\frk m}}

\def\Phi{{\frk n}}
\def\Phi{{\frk N}}
\def\opn#1#2{\def#1{\operatorname{#2}}} 
\opn\chara{char} \opn\length{\ell} \opn\pd{pd} \opn\rk{\lk}\opn\link{link}
\opn\projdim{proj\,dim} \opn\injdim{inj\,dim} \opn\rank{rank}
\opn\depth{depth} \opn\and{and} \opn\grade{grade}
\opn\height{height} \opn\embdim{emb\,dim} \opn\codim{codal}

\opn\Tr{Tr} \opn\bigrank{big\,rank}
\opn\superheight{superheight}\opn\lcm{lcm}
\opn\trdeg{tr\,deg}%
\opn\reg{reg} \opn\lreg{lreg} \opn\ini{in}
\opn\div{div} \opn\Div{Div} \opn\cl{cl} \opn\Cl{Cl}
\opn\Spec{Spec} \opn\Supp{Supp} \opn\supp{supp} \opn\Sing{Sing}
\opn\Ass{Ass} \opn\Min{Min}
\opn\Ann{Ann} \opn\Rad{Rad} \opn\Soc{Soc}
\opn\Im{Im}
 \opn\Ker{Ker} \opn\Coker{Coker} \opn\Am{Am} \opn \inf{inf}
\opn\Hom{Hom} \opn\Tor{Tor} \opn\Ext{Ext} \opn\End{End} \opn\cd{cd}
\opn\Aut{Aut} \opn\id{id}

\opn\nat{nat}
\opn\pff{pf}
\opn\Pf{Pf} \opn\GL{GL} \opn\SL{SL} \opn\mod{mod} \opn\ord{ord}
\opn\cl{cl} \opn\conv{conv} \opn\ext{ext} \opn\rad{rad}
\opn\star{star} \opn\red{red}\opn\H{H}
\opn\aff{aff} \opn\con{conv} \opn\relint{relint} \opn\st{st}
\opn\lk{lk} \opn\cn{cn} \opn\core{core} \opn\vol{vol}
\opn\link{link} \opn\star{star}
\opn\gr{gr}

\def\pot#1#2{#1[\kern-0.28ex[#2]\kern-0.28ex]}

\opn\dirlim{\underrightarrow{\lim}}
\opn\inivlim{\underleftarrow{\lim}}
\let\dirsum=\oplus
\let\tensor=\otimes
\let\iso=\cong

\let\to=\rightarrow

\def\Implies{\ifmmode\Longrightarrow \else
     \unskip${}\Longrightarrow{}$\ignorespaces\fi}
\def\implies{\ifmmode\Rightarrow \else
     \unskip${}\Rightarrow{}$\ignorespaces\fi}
\def\iff{\ifmmode\Longleftrightarrow \else
     \unskip${}\Longleftrightarrow{}$\ignorespaces\fi}

\let\:=\colon
\newtheorem{Theorem}{Theorem}[section]
\newtheorem{Lemma}[Theorem]{Lemma}
\newtheorem{Corollary}[Theorem]{Corollary}
\newtheorem{Proposition}[Theorem]{Proposition}
\newtheorem{Remark}[Theorem]{Remark}

\newtheorem{Example}[Theorem]{Example}

\newtheorem{Definition}[Theorem]{Definition}

\newtheorem{Question}[Theorem]{Question}
\newtheorem{Fact}[Theorem]{Fact}
\let\epsilon\varepsilon
\let\phi=\varphi
\let\kappa=\varkappa
\def\qed{\ifhmode\textqed\fi
   \ifmmode\ifinner\quad\qedsymbol\else\dispqed\fi\fi}
\def\textqed{\unskip\nobreak\penalty50
    \hskip2em\hbox{}\nobreak\hfil\qedsymbol
    \parfillskip=0pt \finalhyphendemerits=0}
\def\dispqed{\rlap{\qquad\qedsymbol}}

\opn\dis{dis}
\def\pnt{{\raise0.5mm\hbox{\large\bf.}}}

\begin{document}
\title{On the structure of sequentially Cohen--Macaulay bigraded modules}

\author{Leila Parsaei Majd,  Ahad Rahimi}

\subjclass[2000]{ 16W50, 13C14, 13D45, 16W70. }
\keywords{  Dimension filtration, Sequentially Cohen--Macaulay, Cohomological dimension, Bigraded modules, Cohen--Macaulay. }
\address{ Leila Parsaei Majd, Department of Mathematics,  Faculty of Science,   Razi University, Baghe abrisham, Kermanshah,
 Iran and, Department of Mathematics, Faculty of Science, Shahid Rajaee Teacher Training University,  P. O. Box: 16785-136, Tehran, Iran.
}\email{leila.parsaei84@yahoo.com}
 \address{ Ahad Rahimi, Department of Mathematics, Faculty of Science, Razi University, Baghe abrisham,  Kermanshah,
 Iran and,
 School of Mathematics, Institute for Research in Fundamental Sciences
(IPM), P. O. Box: 19395-5746, Tehran, Iran.
}\email{ahad.rahimi@razi.ac.ir}
\begin{abstract}
Let $K$ be a field and $S=K[x_1,\ldots,x_m, y_1,\ldots,y_n]$ be the standard bigraded polynomial ring  over $K$. In this paper,  we explicitly describe the structure of finitely generated bigraded "sequentially Cohen--Macaulay" $S$-modules with respect to $Q=(y_1,\ldots,y_n)$. Next, we give a characterization of sequentially Cohen--Macaulay modules with respect to $Q$ in terms of local cohomology modules. Cohen--Macaulay modules that are sequentially Cohen--Macaulay with respect to $Q$ are considered.
\end{abstract}
\maketitle
\section*{Introduction}
Let $K$ be a field and $S=K[x_1,\ldots,x_m,y_1,\ldots,y_n]$  be the standard bigraded $K$-algebra with $\deg x_i=(1,0)$ and $\deg y_j=(0,1)$ for all $i$ and $j$.  We set the bigraded irrelevant ideals $P=(x_1, \dots, x_m)$ and $Q=(y_1, \dots,  y_n)$. Let $M$ be a finitely
 generated bigraded $S$-module. The largest integer $k$ for which $H^k_ Q
(M)\neq 0$, is called the cohomological dimension of $M$ with
respect to $Q$ and denoted by $\cd(Q, M)$.
A finite filtration $
\mathcal{D}:  0=D_0\varsubsetneq D_1
\varsubsetneq
 \dots  \varsubsetneq  D_r=M
 $
 of bigraded submodules of $M$, is called the
dimension filtration of $M$ with respect to $Q$ if $D_{i-1}$ is the
largest bigraded submodule of $D_i$ for which $\cd(Q,D_{i-1})<\cd(Q, D_i)$
for all $i=1, \dots, r$, see  \cite{AR3}.
 In Section 1,  we explicitly describe the structure of the submodules $D_i$ that extends \cite[Proposition 2.2]{S}. In fact,  it is shown that $D_i=\bigcap_{\pp_j\not \in B_{i, Q}}N_j$
  for $ i=1, \ldots, r-1$ where $0 =\bigcap_{j=1}^sN_j $ is a reduced primary decomposition of $0$
in $M$ with $N_j$ is $\pp_j$-primary for $j = 1, \ldots, s$ and
\[
B_{i, Q}=\{ \pp\in \Ass(M): \cd(Q, S/\pp)\leq \cd(Q, D_i)\}.
\]
 In \cite{AR2}, we say  $M$ is Cohen--Macaulay with respect to $Q$, if $\grade(Q,  M)=\cd(Q, M).$ A finite filtration $\mathcal{F}$:
$0=M_0\varsubsetneq M_1 \varsubsetneq
 \dots  \varsubsetneq M_r=M$
 of $M$ by
bigraded submodules $M$, is called a Cohen--Macaulay filtration with respect to $Q$ if
each quotient $M_i/M_{i-1}$ is Cohen--Macaulay with respect to $Q$ and
 \[
 0 \leq \cd(Q, M_1/M_0)<\cd(Q, M_2/M_1)< \dots< \cd(Q, M_r/M_{r-1}).
 \]
 If $M$ admits a
Cohen--Macaulay filtration with respect to $Q$, then we say  $M$ is sequentially Cohen--Macaulay with respect to $Q$, see \cite{AR3}.
   Note that if $M$ is sequentially Cohen--Macaulay with respect to $Q$, then the filtration $\mathcal{F}$ is uniquely determined and it is just the dimension filtration of $M$ with respect to $Q$, that is, $\mathcal{F}=\mathcal{D}$.
      In Section ~2,  we give a characterization of sequentially Cohen--Macaulay modules with respect to $Q$ in terms of local cohomology modules which extends \cite[Corollary 4.4]{CC1} and \cite[Corollary 3.10]{CC2}. We apply this result and the description of the submodules $M_i$ mentioned earlier, showing that $S/I$ is sequentially Cohen--Macaulay with respect to $P$ and $Q$ where $I$ is the Stanley-Reisner ideal that corresponds to the natural triangulation of the projective plane $\PP^2$. Here $S=K[x_1, x_2, x_3, y_1, y_2, y_3]$, $P=(x_1, x_2, x_3)$ and $Q=(y_1, y_2, y_3)$. Note that $S/I$ is Cohen--Macaulay of dimension 3, if $\chara K\neq 2$.

In \cite{AR2} we have shown that if $M$ is a finitely generated bigraded Cohen--Macaulay $S$-module which is Cohen--Macaulay with respect to $P$, then $M$ is Cohen--Macaulay with respect to $Q$. Inspired by this fact and the above example we have the following question: Let $I\subseteq S$ be a monomial ideal. Suppose $S/I$ is Cohen--Macaulay. If $S/I$ is sequentially Cohen--Macauly with respect to $P$, is $S/I$ sequentially Cohen--Macaulay with respect to $Q$?
We do not know the answer of this question yet, however in the last section, we obtain some properties of a Cohen--Macaulay filtration with respect to $Q$ in general provided that the module itself is Cohen--Macaulay, see Propositions~\ref{cm} and  ~\ref{cm1}.  Inspired by Proposition \ref{cm1}, we make the following question: Let  $M$ be a finitely generated bigraded Cohen--Macaulay $S$-module such that $H^k_Q(M)\neq 0$ for all $\grade(Q, M)\leq k \leq \cd(Q, M)$.   Is $H^s_P(M)\neq 0$ for all $\grade(P, M) \leq s \leq \cd(P, M)$?  Of course the question has positive answer in the case that $M$ has only one(two) non-vanishing local cohomology with respect to $Q$. The projective plane $\PP^2$ would also be the case as module with three non-vanishing local cohomology.

\section{The dimension filtration with respect to $Q$}
 Let $K$ be a field and $S=K[x_1,\ldots,x_m,y_1,\ldots,y_n]$  the standard bigraded polynomial ring over $K$. In other words,  $\deg x_i=(1,0)$ and $\deg y_j=(0,1)$ for all $i$ and $j$. We set the bigraded irrelevant ideals
$P=(x_1, \dots, x_m)$ and $Q=(y_1, \dots,  y_n)$, and let $M$ be a
finitely
 generated bigraded $S$-module.
We denote by $\cd(Q, M)$ the cohomological dimension of $M$ with
respect to $Q$ which is the largest integer $i$ for which $H^i_ Q
(M)\neq 0$. Notice that $0\leq \cd(Q, M) \leq n.$

We recall the following facts which will be used in the sequel.

\begin{Fact}
\label{cd1}{\em
\[
\grade(P, M)\leq \dim M-\cd(Q, M),
\]
 and the equality holds if $M$ is Cohen--Macaulay, see \cite[Formula 5]{AR2}.}
\end{Fact}
Let $q\in \ZZ$. In \cite{AR2}, we say $M$ is relative Cohen--Macaulay with respect to $Q$ if $H^i_Q(M) = 0$ for all $i \neq q$. In other words, $\grade(Q, M) = \cd(Q, M) = q$. From now on,  we omit the word "relative" for simplicity and say $M$ is Cohen--Macaulay with respect to $Q$.
\begin{Fact}
\label{cd2}{\em
 If $M$ is Cohen--Macaulay with respect to $Q$ with $ |K|=\infty$, then
\[
\cd(P, M)+\cd(Q, M)=\dim M,
\]
see \cite[Theorem 3.6]{AR2}. }
\end{Fact}

\begin{Fact}{\em
\label{cd3}
The exact sequence  $ 0
\rightarrow M' \rightarrow M \rightarrow M'' \rightarrow 0$ of finitely generated bigraded $S$-modules yields
\[
 \cd(Q,M)=\max\{\cd(Q, M'), \cd(Q,M'')\},
\]
see the general version of \cite[Proposition 4.4]{CJR}.}
\end{Fact}

\begin{Fact}{\em
\label{cd4}
  \[
 \cd(Q,M)  =  \max \{\cd (Q, S/{\pp}): \pp \in \Ass(M)\},
 \]
 see the general version of \cite[Corollary 4.6]{CJR}.}
\end{Fact}

For a finitely generated bigraded $S$-module $M$,  there is a unique largest bigraded submodule $N$ of $M$ for which $\cd(Q, N)<\cd(Q, M)$, see \cite[Lemma 1.9]{AR3}.  We  recall the following definition from \cite{AR3}.
\begin{Definition}
\label{1} {\em  We call a filtration $\mathcal{D}$: $0=D_0\varsubsetneq D_1
\varsubsetneq
 \dots  \varsubsetneq  D_r=M$ of bigraded submodules of $M$ the
dimension filtration of $M$ with respect to $Q$ if $D_{i-1}$ is the
largest bigraded submodule of $D_i$ for which $\cd(Q,D_{i-1})<\cd(Q, D_i)$
for all $i=1, \dots, r$.}
\end{Definition}
\begin{Remark}
\label{F6}{\em
Let $\mathcal{D}$ be the dimension filtration of $M$ with respect to $Q$. For all $i$, the exact sequence $0\to D_{i-1}\to D_i\to D_i/D_{i-1}\to 0$  by using Fact \ref{cd3} yields
\[
\cd(Q, D_i)=\max\{ \cd(Q, D_{i-1}), \cd(Q, D_i/D_{i-1})\}=\cd(Q, D_i/D_{i-1}).
\]
Thus, $\cd(Q, D_{i-1}/D_{i-2})<\cd(Q, D_i/D_{i-1})$ for all $i$.
}
\end{Remark}

Let  $\mathcal{D}$ be the dimension filtration of $M$ with respect to $Q$.  We set
\[
\; \; B_{i, Q}=\{ \pp\in \Ass(M): \cd(Q, S/\pp)\leq \cd(Q, D_i)\},\; \;  I_{i, Q}=\prod_{\pp \in B_{i, Q}}\pp
\]
and
\[
A_{i, Q}=\{ \pp \in \Ass(M): \pp\in V(I_{i, Q})\} \quad \text{for}\quad  i=1, \ldots, r.
\]
\begin{Lemma}
\label{V}
Let the notation be as above. Then the following statements hold
\[
 A_{i, Q}=B_{i, Q}=\Ass(D_i) \quad \text{for}\quad  i=1, \ldots, r.
\]
Consequently,
\[
 \Supp(D_i)\subseteq V(I_{i, Q}) \quad \text{for}\quad  i=1, \ldots, r.
\]
 \end{Lemma}
\begin{proof}
In order to show the first equality, we note that $ B_{i, Q}\subseteq A_{i, Q}$ for $i=1, \ldots, r$. Now let $\pp\in A_{i, Q}$. Then $\pp \in \Ass(M)$ with $I_{i, Q}\subseteq \pp$. Hence $\qq \subseteq \pp$ for some $\qq\in \Ass(M)$  with $ \cd(Q, S/\qq)\leq \cd(Q, D_i)$. The canonical epimorphism $S/\qq \to S/\pp$ yields $\cd(Q, S/\pp)\leq \cd(Q, S/\qq)$ by Fact \ref{cd3}. It follows that $\pp\in B_{i, Q}$ and hence  $A_{i, Q} \subseteq B_{i, Q}.$

To show the second equality, let $\pp \in B_{i, Q}$.  Then there is a submodule $N\subseteq M$ such that $N\iso S/\pp$ and $\cd(Q, S/\pp)\leq \cd(Q, D_i)$. Using Fact \ref{cd3} we have
\[
\cd(Q, N+D_i)=\max\{ \cd(Q, D_i), \cd(Q, N/(N\cap D_i))\}=\cd(Q, D_i),
\]
and hence $N\subseteq D_i$. This shows $\pp \in\Ass(D_i)$ and therefore $B_{i,Q}\subseteq \Ass(D_i)$. Now let $\pp \in \Ass(D_i).$ Then $\pp \in \Ass(M)$ and  $\cd(Q, S/\pp)\leq \cd(Q, D_i)$ by Fact \ref{cd4}. This shows $\pp \in B_{i, Q}$ and hence $\Ass(D_i)\subseteq B_{i, Q}.$

\end{proof}
In the following we describe the structure of the submodules $D_i$ in the dimension filtration of $\mathcal{D}$ with respect to $Q$
which extends \cite[Proposition 2.2]{S}.
\begin{Proposition}
\label{Decomposition0}
Let $\mathcal{D}$ be the dimension filtration of $M$ with respect to $Q$.  Then
 \[
 D_i=H^0_{I_{i, Q}}(M)=\bigcap_{\pp_j\not \in B_{i, Q}}N_j
 \]
 for $ i=1, \ldots, r-1$ where $0 =\bigcap_{j=1}^sN_j $ is a reduced primary decomposition of $0$
in $M$ with $N_j$ is $\pp_j$-primary for $j = 1, \ldots, s$.
\end{Proposition}
\begin{proof}
In order to prove the first equality,  we have $V(\Ann(D_i))=\Supp(D_i)\subseteq V(I_{i, Q})$ for $ i=1, \ldots, r-1$ by Lemma \ref{V}.
 Since $I_{i, Q}$ is finitely generated, it follows that $I_{i, Q}^{k_i}\subseteq \Ann(D_i)$ for some integer $k_i$ and hence  $I_{i, Q}^{k_i}D_i=0$ for some $k_i$. Thus  $D_i=H^0_{I_{i, Q}}(D_i)\subseteq H^0_{I_{i, Q}}(M)$ for $ i=1, \ldots, r-1$.

 Now we prove the equality by decreasing induction on $i$.  For $i=r-1$, we assume that $D_{r-1}\varsubsetneq H^0_{I_{r-1, Q}}(M)\subseteq D_r=M$. It follows from the definition dimension filtration that $\cd(Q,H^0_{I_{r-1, Q}}(M))=\cd(Q, M).$
Note that
\[
\Ass H^0_{I_{i, Q}}(M)=A_{i,Q}=\Ass(D_i) \quad \text{for}\quad  i=1, \ldots, r-1
\]
 by \cite[Proposition 3.13]{Ei}(c) and Lemma \ref{V}. It follows that $\cd(Q,H^0_{I_{r-1, Q}}(M)) =\cd(Q, D_{r-1, Q})$, and hence $\cd(Q, D_{r-1, Q})=\cd(Q, M)$, a contradiction. Thus $D_{r-1, Q}=H^0_{I_{r-1, Q}}(M)$.
Now let $1<i<r-1$, and assume that $D_i=H^0_{I_{i, Q}}(M)$. We show  $D_{i-1}=H^0_{I_{i-1, Q}}(M)$. Assume  $D_{i-1}\varsubsetneq H^0_{I_{i-1, Q}}(M)$. As   $H^0_{I_{i-1, Q}}(M)\subseteq H^0_{I_{i, Q}}(M)=D_{i}$, we have  $\cd(Q, H^0_{I_{i-1, Q}}(M))\geq \cd(Q, D_{i})$. Since $\Ass H^0_{I_{i-1, Q}}(M)=\Ass(D_{i-1})$, it follows that $\cd(Q, D_{i-1})=\cd(Q, H^0_{I_{i-1, Q}}(M))\geq \cd(Q, D_i)$, a contradiction. Therefore, $D_{i-1}=H^0_{I_{i-1, Q}}(M)$. The second equality follows from Lemma \ref{V} and \cite[Proposition 3.13]{Ei}(a).
\end{proof}
\begin{Remark}
\label{unmixed}{\em
Let $\mathcal{D}$ be the dimension filtration of $M$ with respect to $Q$ with $\cd(Q, M)=q$. We call the submodule
\[
D_{r-1}=\bigcap_{\pp_j \not \in B_{r-1,Q}}N_j=\bigcap_{\cd(Q, S/\pp_j)=q}N_j,
\]
the {\em unmixed component of $M$ with respect to $Q$}  and denote it by $u_{ Q, M}(0)$. Notice that $u_{\mm, M}(0)=u_M(0)$ introduced by Schenzel in \cite{S}.
If  $M$ is relatively unmixed with respect to $Q$, that is,  $\cd(Q, M)=\cd(Q, S/\pp)$ for all $\pp \in \Ass(M)$, then by Proposition \ref{Decomposition0}  we have
\[
D_i=\bigcap_{\pp_j\not \in B_{i, Q}}N_j=\bigcap_{j=1}^sN_j=0 \quad \text{for all}\quad i<r.
\]
}
\end{Remark}

\begin{Corollary}
\label{ass}
Let $\mathcal{D}$ be the dimension filtration of $M$ with respect to $Q$. Then for $i=1, \dots, r$ we have
\[
\Ass(M/D_i)=\Ass(M)- \Ass(D_{i}).
\]
\end{Corollary}
\begin{proof}
 The assertion follows from Proposition \ref{Decomposition0}, Lemma \ref{V} and the fact that $\Ass M/H^0_{I_{i, Q}}(M)=\Ass(M)-A_{i, Q}$, see \cite[Proposition 3.13]{Ei}(c).
\end{proof}

\section{ Sequentially Cohen--Macaulay with respect to $Q$}
We recall the following definition from \cite{AR3}.
\begin{Definition}
\label{seq1}{\em Let $M$ be a finitely generated bigraded
$S$-module. We call a finite filtration $\mathcal{F}$:
$0=M_0\varsubsetneq M_1 \varsubsetneq
 \dots  \varsubsetneq M_r=M$
 of $M$ by
bigraded submodules $M$ a Cohen--Macaulay filtration with
respect to $Q$ if
 \begin{itemize}
\item[{(a)}] Each quotient $M_i/M_{i-1}$ is Cohen--Macaulay with respect to $Q$;
 \item[{(b)}] $0 \leq \cd(Q, M_1/M_0)<\cd(Q, M_2/M_1)< \dots< \cd(Q, M_r/M_{r-1})$.
 \end{itemize}
 We call $M$ to
be {\em sequentially Cohen--Macaulay with respect to $Q$} if $M$ admits a
Cohen--Macaulay filtration with respect to $Q$.
   }
\end{Definition}
Note that if $M$ is sequentially Cohen--Macaulay with respect to $Q$, then the filtration $\mathcal{F}$ in the definition above is uniquely determined and it is just the dimension filtration of $M$ with respect to $Q$ defined in Definition \ref{1}, see \cite[Proposision 1.12]{AR3}.

We have the following characterization of sequentially Cohen--Macaulay modules with respect to $Q$ in terms of local cohomology modules which extends \cite[Corollary 4.4]{CC1} and \cite[Corollary 3.10]{CC2}.
\begin{Proposition}
\label{equivalent}
Let $\mathcal{D}$:
$0=D_0\varsubsetneq D_1 \varsubsetneq
 \dots  \varsubsetneq D_r=M$ be the dimension filtration of $M$ with respect to $Q$. Then the
following statements are equivalent:
 \begin{itemize}
\item[{(a)}]  $M$ is sequentially Cohen--Macaulay with respect to $Q$;
\item[{(b)}] $H^k_Q(M/D_{i-1})=0$ for $i=1, \dots, r$ and  $k<\cd(Q, D_i)$;
\item[{(c)}] $\grade(Q, M/D_{i-1})=\cd(Q, D_i)$ for $i=1, \dots, r$.
\end{itemize}
\end{Proposition}
\begin{proof}
$(a)\Rightarrow (b)$: We proceed by decreasing induction on $i$. As $D_i/D_{i-1}$ is Cohen--Macaulay with respect to $Q$ for all $i$, thus for $i=r$ we have $H^k_Q(M/D_{r-1})=0$ for  $k<\cd(Q, M)$. Now let $1<i<r$, and assume that  $H^k_Q(M/D_{i-1})=0$ for  $k<\cd(Q, D_i)$.
The exact sequence
\[
0\to D_{i-1}/D_{i-2}\to M/D_{i-2}\to M/D_{i-1}\to 0,
\]
induces the following long exact sequence
\begin{eqnarray}
\label{character}
  \cdots \to H^k_Q(D_{i-1}/D_{i-2})\to H^k_Q(M/D_{i-2})\to H^k_Q(M/D_{i-1})\to \cdots.
\end{eqnarray}
 As $ D_{i-1}/D_{i-2}$ is Cohen--Macaulay with respect to $Q$, we have $ H^k_Q(D_{i-1}/D_{i-2})=0$   for  $k<\cd(Q, D_{i-1})$.
  By Remark \ref{F6},  we have $\cd(Q, D_{i-1})=\cd(Q, D_{i-1}/D_{i-2})<\cd(Q, D_i).$ Hence by using (\ref{character}) and the induction hypothesis, we have  $H^k_Q(M/D_{i-2})=0$ for $k<\cd(Q, D_{i-1})$, as desired.

 $(b) \Rightarrow (a)$:  By Remark \ref{F6} we have $\cd(Q, D_i/D_{i-1})<\cd(Q, D_{i+1}/D_i)$ for all $i$. Thus it suffices to show that  $D_i/D_{i-1}$ is Cohen--Macaulay with respect to $Q$ for all $i$. We prove this statement by decreasing induction on $i$. In condition (b), we first assume $i=r$. It follows that  $M/D_{r-1}$ is Cohen--Macaulay with respect to $Q$. Now let $1<i<r$, and assume that $D_i/D_{i-1}$ is Cohen--Macaulay with respect to $Q$.  The exact sequence
\[
0\to D_{i}/D_{i-1}\to M/D_{i-1}\to M/D_{i}\to 0,
\]
induces the following long exact sequence
\begin{eqnarray}
\label{character2}
  \cdots \to H^{k-1}_Q(D_{i}/D_{i-1})\to H^{k-1}_Q(M/D_{i-1})\to H^{k-1}_Q(M/D_{i})\to\cdots.
\end{eqnarray}
 Suppose  $k<\cd(Q, D_{i-1})$. Induction hypothesis and our assumption say that $ H^{k-1}_Q(D_{i}/D_{i-1})=H^{k-1}_Q(M/D_{i})=0$. Hence  $H^{k-1}_Q(M/D_{i-1})=0$  by (\ref{character2}).  We have $H^k_Q(M/D_{i-2})=0$ for $k<\cd(Q, D_{i-1})$ because of our assumption again. Thus  $H^k_Q(D_{i-1}/D_{i-2})=0$ for $k<\cd(Q, D_{i-1})$ by (\ref{character}). Therefore $D_{i-1}/D_{i-2}$ is Cohen--Macaulay with respect to $Q$, as desired.

 $(b)\Rightarrow (c)$: We set $\cd(Q, D_i)=\cd(Q, D_i/D_{i-1})=q_i$ for $i=1, \dots, r$. Our assumption says that $\grade(Q, M/D_{i-1})\geq q_i$ for $i=1, \dots, r$. We only need to know $H^{q_i}_Q(M/D_{i-1})\neq 0$. Consider the long exact sequence
 \begin{eqnarray}
\label{character3}
  \cdots \to H^{q_i-1}_Q(M/D_{i}) \to  H^{q_i}_Q(D_{i}/D_{i-1})\to H^{q_i}_Q(M/D_{i-1})\to \cdots.
\end{eqnarray}
Since $q_i-1<q_i<q_{i+1}$, it follows from our assumption that $H^{q_i-1}_Q(M/D_{i})=0$. If $H^{q_i}_Q(M/D_{i-1}) =0$, then by (\ref{character3}) we have  $H^{q_i}_Q(D_{i}/D_{i-1})=0$, a contradiction.
 The implication $(c)\Rightarrow (b)$ is obvious.
\end{proof}
As an application of Proposition \ref{Decomposition0} and Proposition \ref{equivalent} we have

\begin{Example}
\label{projective plane}{\em
Let $I$ be the Stanley-Reisner ideal that corresponds to the natural triangulation of the projective plane $\PP^2$. Then
\[
I=(x_1x_2x_3, x_1x_2y_1, x_1x_3y_2, x_1y_1y_3, x_1y_2y_3, x_2x_3y_3, x_2y_1y_2,x_2y_2y_3, x_3y_1y_2,x_3y_1y_3).
\]
We set $R=S/I$ where  $S=K[x_1, x_2, x_3, y_1, y_2, y_3]$, $P=(x_1, x_2, x_3)$ and $Q=(y_1,  y_2, y_3)$.
Our aim is to show that $R$ is sequentially Cohen--Macaulay with respect to $P$ and $Q$.
Note that $R$ is Cohen--Macaulay of dimension 3 if $\chara K\neq 2$.
The ideal $I$ has the minimal primary decomposition $I=\bigcap_{i=1}^{10}\pp_i$ where
$\pp_1=(x_3, y_1, y_3), \pp_2=(x_1, y_1, y_3), \pp_3=(x_2, y_1, y_2), \pp_4=(x_3, y_1, y_2), \pp_5=(x_1, y_2, y_3), \pp_6=(x_2, y_2, y_3), \pp_7=(x_2, x_3, y_3), \pp_8=(x_1, x_2, y_1), \pp_9=(x_1, x_3, y_2),  \pp_{10}=(x_1, x_2, x_3)$. As  $P=\pp_{10}\in \Ass(R)$, we have $\grade(P, R)=0$. By Fact \ref{cd4} we have $\cd(P, R)=2$  and $\cd(Q, R)=3$.
Since $R$ is Cohen--Macaulay, it follows from Fact \ref{cd1} that $\grade(Q, R)=1$. We first show that $R$ is sequentially Cohen--Macaulay with respect to $P$.
By Proposition \ref{Decomposition0},  $R$ has the dimension filtration
\[
0=R_0\varsubsetneq R_1 \varsubsetneq R_2 \varsubsetneq R_3=R,
 \]
 with respect to $P$ where
 \[
 R_1=\bigcap_{i=1}^{9}\pp_i/I \quad \text {and} \quad  R_2=\bigcap_{i=1}^{6}\pp_i/I.
 \]
  By Corollary \ref{ass} we have
  \[
 \Ass(R_1)=\Ass(R)-\Ass(R/R_1)=\{ \pp_{10} \}.
 \]
 and
 \[
 \Ass(R_2)=\Ass(R)-\Ass(R/R_2)=\{\pp_7, \pp_8, \pp_9, \pp_{10} \}.
 \]
  It follows that $\cd(P, R_1)=0$ and $\cd(P, R_2)=1$. We set $I_1=\bigcap_{i=1}^{9}\pp_i$ and $I_2=\bigcap_{i=1}^{6}\pp_i$. In view of Proposition \ref{equivalent}, we need to show
 \[
 \grade(P, R_3/R_0)=\grade(P, R)=\cd(P, R_1)=0,
 \]

\[
\grade(P, R_3/R_1)=\grade(P, S/I_1)=\cd(P, R_2)=1
 \]
  and
 \[
 \grade(P, R_3/R_2)=\grade(P, S/I_2)=\cd(P, R)=2.
 \]
 The first equality is obvious. As $P\not\subseteq \pp_i$ for $i=1, \dots, 9$, we have $\grade(P, S/I_1)\geq 1.$ On the other hand, $\grade(P, S/I_1)\leq \dim S/I_1-\cd(Q, S/I_1)=3-2=1.$ Thus the second equality holds.  In order to show the third equality, we note that $S/I_2$ has dimension 3 and by using CoCoA \cite{Co} depth 2. Thus Fact \ref{cd1} can not be used to compute $\grade(P, S/I_2)$.  We set $\qq_1=\pp_1\cap \pp_2=(x_1x_3, y_1, y_3)$, $\qq_2=\pp_3\cap \pp_4=(x_2x_3, y_1, y_2)$ and $\qq_3=\pp_5\cap \pp_6=(x_1x_2, y_2, y_3)$.
Consider the exact sequence
\[
0\to S/\qq_1\cap \qq_2\to S/\qq_1\dirsum S/\qq_2 \to S/(\qq_1+\qq_2)\to 0.
 \]
 Since $\grade(P, S/\qq_1\dirsum S/\qq_2)=2$ and $\grade(P,S/(\qq_1+\qq_2))=1 $, it follows that $\grade(P, S/(\qq_1\cap \qq_2))\geq 2$. As  $\cd(P, S/(\qq_1\cap \qq_2))= 2$,  we have $\grade(P, S/(\qq_1\cap \qq_2))= 2$.
Consider the exact sequence
\begin{eqnarray}
\label{F7}
0\to S/I_2\to S/\qq_1\cap \qq_2\dirsum S/\qq_3 \to S/(\qq_1+\qq_3)\cap (\qq_2+\qq_3)\to 0.
\end{eqnarray}
The exact sequence
\[
0\to S/(\qq_1+\qq_3)\cap (\qq_2+\qq_3)\to  S/(\qq_1+\qq_3)\dirsum  S/(\qq_2+\qq_3)\to S/(\qq_1+\qq_2+\qq_3)\to 0
 \]
 yields that $\grade(P, S/(\qq_1+\qq_3)\cap (\qq_2+\qq_3))\geq 1.$ Hence by (\ref{F7}) we have $\grade(P, S/I_2)\geq 2$. As $\cd(P, S/I_2)=2$, we conclude that $\grade(P, S/I_2)= 2$, as desired.

 Next, we show that $R$ is sequentially Cohen--Macaulay  with respect to $Q$. By Proposition \ref{Decomposition0},   $R$ has the dimension filtration $0=R_0\varsubsetneq R_1 \varsubsetneq R_2 \varsubsetneq R_3=R$ with respect to $Q$ where $R_1=\bigcap_{i=7}^{10}\pp_i/I$ and $R_2=\pp_{10}/I$.  By Corollary \ref{ass} we have $\cd(Q, R_1)=1$ and $\cd(Q, R_2)=2$.  We set $J=\bigcap_{i=7}^{10}\pp_i$.
In view of Proposition \ref{equivalent},  we need to show
 \[
 \grade(Q, R_3/R_0)=\grade(Q, R)=\cd(Q, R_1)=1,
 \]

 \[
 \grade(Q, R_3/R_1)=\grade(Q, S/J)=\cd(Q, R_2)=2
 \]
 and
\[
\grade(Q, R_3/R_2)=\grade(Q, S/\pp_{10})=\cd(Q, R)=3.
 \]
  The first and the third statements are obvious. In order to prove the second equality, consider the exact sequence
\begin{eqnarray}
\label{F8}
0\to S/J\to S/\cap_{i=7}^{9}\pp_i\dirsum S/\pp_{10}\to S/\cap_{i=7}^{9}(\pp_i+\pp_{10})\to 0.
\end{eqnarray}
An exact sequence argument shows that
\[
\grade(Q, S/\cap_{i=7}^{9}\pp_i)=\grade(Q, S/\cap_{i=7}^{9}(\pp_i+\pp_{10}))=2.
\]
Thus it follows from (\ref{F8}) that $\grade(Q, S/J)\geq 2$. On the other hand,
\[
\grade(Q, S/J)\leq \dim S/J-\cd(P, S/J)=3-1=2.
\]
 Therefore, $\grade(Q, S/J)= 2$, as desired.
}
\end{Example}

\section{Cohen--Macaulay modules that are sequentially Cohen--Macaulay with respect to $Q$}
In \cite{AR2} we have shown that if $M$ is a finitely generated bigraded Cohen--Macaulay $S$-module which is Cohen--Macaulay with respect to $P$, then $M$ is Cohen--Macaulay with respect to $Q$. Inspired by this fact and Example \ref{projective plane} we have the following question
\begin{Question}
Let $I\subseteq S$ be a monomial ideal. Suppose $S/I$ is Cohen--Macaulay. If $S/I$ is sequentially Cohen--Macauly with respect to $P$, is $S/I$ sequentially Cohen--Macaulay with respect to $Q$?
\end{Question}
We do not know the answer of this question yet, however in this section, we obtain some properties of a Cohen--Macaulay filtration with respect to $Q$ in general provided that the module itself is Cohen--Macaulay.

\begin{Fact}
\label{cd5}{\em
For a Cohen--Macaulay filtration $\mathcal{F}$ with respect to $Q$ we recall the following fact from \cite[Corollary 1.8]{AR3}
\[
\grade(Q, M_i)=\grade(Q, M) \quad \text{for}\quad   i=1, \dots, r.
\]}
\end{Fact}

\begin{Proposition}
\label{cm}
Let $M$ be a finitely generated bigraded Cohen--Macaulay $S$-module with $ |K|=\infty$. Suppose $M$ is sequentially Cohen--Macaulay with respect to $Q$ with the Cohen--Macaulay filtration
$0=M_0\varsubsetneq M_1 \varsubsetneq
 \dots  \varsubsetneq M_r=M$ with respect to $Q$. Then
  \begin{itemize}
\item[{(a)}]   $\cd(P, M_i)=\cd(P, M)$ for $i=1, \dots, r.$
\item[{(b)}]  $\grade(Q, M_i)+\cd(P, M_i)=\dim M_i$ for $i=1, \dots, r.$
\end{itemize}
\end{Proposition}
\begin{proof}
 In order to prove (a), since $M_1$ is Cohen--Macaulay with respect to $Q$, it follows from Fact \ref{cd2} that $\cd(P, M_1)+\cd(Q, M_1)=\dim M_1.$ By Fact \ref{cd5} we have $\cd(Q, M_1)=\grade(Q, M_1)=\grade(Q, M).$  Since $M$ is Cohen--Macaulay, it follows from ~\cite[Lemma 1.11]{AR3} that $\dim M_1=\dim M$ and $\cd(P, M)=\dim M-\grade(Q, M)$ by Fact \ref{cd1}. Thus we conclude that $\cd(P, M_1)=\cd(P, M).$ As by Fact \ref{cd3} we have $\cd(P, M_{i-1})\leq \cd(P, M_{i})$ for all $i$, the first equality follows.

 For the proof (b),   by \cite[Lemma 1.11]{AR3} we have $\dim M_i=\dim M$  for $i=1, \dots, r$. Thus the second equalities  follow from Fact \ref{cd1}, Fact \ref{cd5}  and part (a).
 \end{proof}

\begin{Proposition}
\label{cm1}
Let the assumptions and the notation be as in Proposition \ref{cm}. Then the following statements are equivalent:
\begin{itemize}
\item[{(a)}]  $\cd(P, M)+\cd(Q, M)=\dim M +r-1$;
\item[{(b)}]  $H^s_Q(M)\neq 0$ for all $\grade(Q, M)\leq s\leq \cd(Q, M).$
\end{itemize}
\begin{proof}
We first assume that $r=1$. As $M$ is Cohen--Macaulay,  by Fact \ref{cd1} and Fact \ref{cd2} we have  $\cd(P, M)+\cd(Q, M)=\dim M$ if and only if  $M$ is Cohen--Macaulay with respect to $Q$. Thus the claim holds in this case. Now let $r\geq 2$.  By Fact \ref{cd1} we have $\cd(P, M)+\cd(Q, M)=\dim M +r-1$ if and only if $\cd(Q, M)-\grade(Q, M)=r-1$. This is equivalent to saying that $\cd(Q, M_{i+1})=\cd(Q, M_i)+1$ for $i=1, \dots, r-1$ by Fact \ref{cd5}. By \cite[Proposition 1.7]{AR3} this is equivalent to saying that $H^s_Q(M)\neq 0$ for all $\grade(Q, M)\leq s\leq \cd(Q, M).$
\end{proof}

\end{Proposition}
The following example shows that the condition that ''$M$ is Cohen--Macaulay" is required for Proposition \ref{cm1}.
\begin{Example}{\em
We set $K[x]=K[x_1, \dots, x_m]$ and $K[y]=K[y_1, \dots, y_n]$. Let $L$ be a non-zero finitely generated graded $K[x]$-module of depth 0 and dimension 1, and
 $N$ a non-zero finitely generated graded $K[y]$-module of depth 0 and dimension 1. We set $M=L\tensor_KN$ and consider it as $S$-module. One has $\depth M=0$ and $\dim M=2$. Hence $M$ is not Cohen--Macaulay.
 On the other hand,  $\grade(Q, M)=\depth N=0$ and $\cd(Q, M)=\dim N=1=\dim L=\cd(P, M)$. Hence $M$ is sequentially Cohen--Macaulay with respect to $Q$ which satisfies condition (b) in Proposition \ref{cm1}, while the equality (a) does not hold.}
\end{Example}
The following question is inspired by Proposition \ref{cm1}.
\begin{Question}
\label{nonvanishing} {\em
Let $M$ be a finitely generated bigraded Cohen--Macaulay $S$-module such that $H^k_Q(M)\neq 0$ for all $\grade(Q, M)\leq k\leq \cd(Q, M).$ Is $H^s_P(M)\neq 0$ for all $\grade(P, M)\leq s\leq \cd(P, M)?$ }
\end{Question}
\begin{Remark}{\em
Of course the question has positive answer in the following cases, namely, if $M$ has only one(two) non-vanishing local cohomology with respect to $Q$. This immediately follows by Fact \ref{cd1}. The projective plane $\PP^2$  given in Example \ref{projective plane} is also the case as module with three non-vanishing local cohomology. }
\end{Remark}

\end{document}